\documentclass[reqno, 12pt]{amsart}
\makeatletter
\let\origsection=\section \def\section{\@ifstar{\origsection*}{\mysection}} 
\def\mysection{\@startsection{section}{1}\z@{.7\linespacing\@plus\linespacing}{.5\linespacing}{\normalfont\scshape\centering\S}}
\makeatother        

\usepackage{amsmath,amssymb,amsthm}
\usepackage{mathrsfs}
\usepackage{mathabx}\changenotsign

\usepackage{xcolor}
\usepackage[backref]{hyperref}
\hypersetup{
    colorlinks,
    linkcolor={red!60!black},
    citecolor={green!60!black},
    urlcolor={blue!60!black}
}

\usepackage[abbrev,msc-links,backrefs]{amsrefs}
\usepackage[T1]{fontenc}
\usepackage{lmodern}

\usepackage[english]{babel}

\usepackage{geometry}
\usepackage{enumitem}
\def\rmlabel{\upshape({\itshape \roman*\,})}

\let\polishlcross=\l
\def\l{\ifmmode\ell\else\polishlcross\fi}

\let\emptyset=\varnothing
\let\setminus=\smallsetminus

\makeatletter
\def\moverlay{\mathpalette\mov@rlay}
\def\mov@rlay#1#2{\leavevmode\vtop{   \baselineskip\z@skip \lineskiplimit-\maxdimen
   \ialign{\hfil$\m@th#1##$\hfil\cr#2\crcr}}}
\newcommand{\charfusion}[3][\mathord]{
    #1{\ifx#1\mathop\vphantom{#2}\fi
        \mathpalette\mov@rlay{#2\cr#3}
      }
    \ifx#1\mathop\expandafter\displaylimits\fi}
\makeatother

\newtheorem{theorem}             {Theorem}[section]
\newtheorem{lemma}     	[theorem] {Lemma}        
   
\newtheorem{property}  	[theorem] {Property}   
\newtheorem{definition}	[theorem] {Definition}   
\newtheorem{proposition}[theorem] {Proposition}   
\newtheorem{corollary}	[theorem] {Corollary}
\newtheorem{fact}	[theorem] {Fact}

\let\:=\colon

\def\calq{{\mathcal Q}}

\let\epsilon\varepsilon
\let\eps\varepsilon

\def\ind{\mathop{\text{\rm ind}}\nolimits}

\def\BDD{\mathop{\text{\rm BDD}}\nolimits}
\def\TUPLE{\mathop{\text{\rm TUPLE}}\nolimits}
\def\ni{\mathop{\neg\text{\rm ind}}\nolimits}
\def\clean{\mathop{\delta\text{\rm -clean}}\nolimits}
\def\poll{\mathop{\delta\text{\rm -poll}}\nolimits}
\def\stb{\mathop{\text{\rm stb}}\nolimits}
\def\set{\mathop{\text{\rm set}}\nolimits}

\def\ord{\mathop{\text{\rm ord}}\nolimits}

\begin{document}

\title{Counting results for sparse pseudorandom hypergraphs~I}

\author[Y.~Kohayakawa]{Yoshiharu Kohayakawa}

\author[G.~O.~Mota]{Guilherme Oliveira Mota}
\address{Instituto de Matem\'atica e Estat\'{\i}stica, Universidade de S\~ao Paulo, S\~ao Paulo, Brazil}
\email{\{\,yoshi\,|\,mota\,\}@ime.usp.br}

\author[M.~Schacht]{Mathias Schacht}
\address{Fachbereich Mathematik, Universit\"at Hamburg, Hamburg, Germany}
\email{schacht@math.uni-hamburg.de}

\author[A.~Taraz]{Anusch Taraz}
\address{Institut f\"ur Mathematik, Technische Universit\"at Hamburg--Harburg, Hamburg, Germany}
\email{taraz@tuhh.de}

\thanks{Y.~Kohayakawa was partially supported by FAPESP (2013/03447-6, 2013/07699-0), 
	CNPq (310974/2013-5, 459335/2014-6), NUMEC/USP (Project MaCLinC/USP) and the 
	NSF (DMS~1102086). 
	G.~O.~Mota was supported by FAPESP (2009/06294-0, 2013/11431-2, 2013/20733-2). 
	M.~Schacht was supported by the Heisenberg-Programme of the DFG (grant SCHA 1263/4-1). 
	A.~Taraz was supported in part by DFG grant TA 309/2-2.
	The cooperation was supported by a joint CAPES/DAAD PROBRAL (333/09 and 430/15).}

\keywords{Hypergraphs, Counting lemma, pseudorandomness}
\subjclass[2010]{05C60 (primary), 05C65 (secondary)}

\begin{abstract}
We establish a so-called counting lemma that allows embeddings of certain linear uniform 
hypergraphs into sparse pseudorandom hypergraphs, generalizing a result for graphs [Embedding graphs with bounded degree in 
sparse pseudo\-random graphs, Israel J. Math. 139 (2004), 93--137]. 
Applications of our result are presented in the companion paper [Counting results for sparse pseudorandom hypergraphs II].
\end{abstract}

\maketitle

\section{Introduction}
Many problems in extremal combinatorics concern embeddings of graphs and hypergraphs of fixed isomorphism type into a large host 
graph/hypergraph. The systematic study of pseudorandom graphs was initiated by Thomason~\cites{Th87a,Th87b} and since then many 
embedding results have been developed for host pseudorandom graphs. For example, a well-known consequence of the 
Chung--Graham--Wilson theorem~\cite{ChGrWi88} asserts that dense pseudorandom graphs $G$ contain the ``right'' number of copies 
of any 
fixed graph, where ``right'' means approximately  the same number of copies as expected in a random graph with the same density 
as $G$. In view of this result, the question arises to which extent it can be generalized to sparse pseudorandom graphs, and 
results in this direction can be found in~\cites{ChGr02,ChGr08,CoFoZh13+,KoRoSi04}. We continue this line of research for 
embedding properties of sparse pseudorandom hypergraphs.
Counting lemmas for pseudorandom hypergraphs were also investigated by Conlon, Fox and Zhao~\cite{CoFoZh13+b}.

Let $G=(V,E)$ be a $k$-uniform hypergraph.
For every $1\leq i \leq k-1$ and every $i$-element set $\{x_1,\ldots,x_{i}\}\in {V\choose i}$, let
\begin{equation*}
N_G(x_1,\ldots,x_{i})=\left\{\{x_{i+1},\ldots,x_k\}\in{V\choose k-i}\colon \{x_1,\ldots,x_{k}\}\in 
E\right\},
\end{equation*}
i.e., $N_G(x_1,\ldots,x_{i})$ is the set of elements of ${V\choose k-i}$ that form an edge of $G$ 
together 
with $\{x_1,\ldots,x_{i}\}$. In what follows, our hypergraphs will usually be $k$-uniform and will have $n$ vertices. The 
parameters $n$ and $k$ will often be omitted if there is no danger of confusion.

\begin{property}[Boundedness Property]
Let $k\geq 2$. We define $\BDD(d,C,p)$ as the family of $n$-vertex 
$k$-uniform hypergraphs 
$G=(V,E)$ such that, for all $1\leq r\leq d$ and all families of distinct sets $S_1,\ldots,S_r\in 
{V\choose k-1}$, we have
 \begin{equation}  |N_G(S_1)\cap\ldots\cap N_G(S_r)|  \leq Cnp^r.
 \end{equation}
\end{property}

\begin{property}[Tuple Property]
Let $k\geq 2$. We define $\TUPLE(d,\delta,p)$ as the family of $n$-vertex 
$k$-uniform hypergraphs 
$G=(V,E)$ such that, for all $1\leq r\leq d$, the following holds.
\begin{equation}\big||N_G(S_1)\cap\ldots\cap N_G(S_r)|-np^r\big|<\delta np^r
\end{equation}
for all but at most $\delta{{n\choose k-1}\choose r}$ families $\{S_1,\ldots,S_r\}$ of
$r$ distinct sets of ${V\choose k-1}$.
\end{property}

The notion of pseudorandomness considered in this paper is given in Definition~\ref{def:pseudo} below.

\begin{definition}\label{def:pseudo}
A $k$-uniform hypergraph $G=(V,E)$ is \emph{$(d_1,C,d_2,\delta,p)$-pseudorandom} if $|E|=p{n\choose k}$ and $G$ 
satisfies 
$\BDD(d_1,C,p)$ and $\TUPLE(d_2,\delta,p)$.
\end{definition}

Note that property $\TUPLE$ implies edge-density close to $p$, but we 
put the condition $|E|=p{n\choose k}$ in the definition of pseudorandomness for 
convenience. We remark that similar notions of pseudorandomness in hypergraphs were 
considered in~\cites{FrKr12,FrKrLo12}.

Our main result, Theorem~\ref{thm:emb} below, estimates the number of copies of some linear $k$-uniform 
hypergraphs in sparse pseudorandom hypergraphs.
An \emph{embedding} of a hypergraph $H$ into a hypergraph $G$ is an injective mapping $\phi\colon V(H)\to 
V(G)$ such that
$\{\phi(v_1),\ldots,\phi(v_k)\}\in 
E(G)$ whenever $\{v_1,\ldots,v_k\}\in E(H)$.
An edge~$e$ of a linear $k$-uniform hypergraph $E(H)$ is called \emph{connector} if there exist $v\in 
V(H)\setminus e$ and $k$ edges $e_1,\ldots,e_k$ containing $v$ such that $|e\cap e_i|=1$ for $1\leq i\leq k$. Note 
that, for $k=2$, a connector is an edge that is contained in a triangle. Moreover, since 
$H$ is linear, $e\cap e_i\neq e\cap e_j$ for all $1\leq i<j\leq k$.
Given a $k$-uniform hypergraph $H$, let
\begin{equation*}
d_H=\max\{\delta(J)\colon J\subset H\}\,\text{ and }\,
D_H=\min\{k d_H,\Delta(H)\},
\end{equation*}
where $\delta(J)$ and $\Delta(J)$ stand, respectively, for the minimum and maximum 
degree of a vertex in $V(J)$. Note that $d_H\leq D_H$.

Kohayakawa, R\"odl and Sissokho \cite{KoRoSi04} proved the following counting lemma: 
{\sl given a fixed triangle-free graph $H$ and $p=p(n)\gg n^{-1/D_H}$ with $p=o(1)$, for all
$\varepsilon>0$ and $C>1$, there exists $\delta>0$ such that, if $G$ is 
an $n$-vertex $(D_H,C,2,\delta,p)$-pseudorandom graph and $n$ is sufficiently large, then
\begin{equation*}
{\big||\mathcal{E}(H,G)|-n^{v(H)}p^{e(H)}\big|<\varepsilon n^{v(H)}p^{e(H)}},
\end{equation*}
where $\mathcal{E}(H,G)$ stands for the set of all embeddings from $H$ into $G$. }
Our main theorem generalizes this result for $k$-uniform hypergraphs. 

\begin{theorem}\label{thm:emb}
Let $k\geq 2$ and $m\geq 4$ be integers and let 
$\eps>0$ and $C>1$ be fixed. Let~$H$ be a linear $k$-uniform 
connector-free hypergraph on 
$m$ vertices. Then there exists $\delta>0$ for which 
the following holds for any $p=p(n)$ with $p\gg n^{-1/D_H}$ and $p=o(1)$ and for any sufficiently large $n$. 

If $G$ is an $n$-vertex $k$-uniform hypergraph that is 
$(D_H,C,2,\delta,p)$-pseudorandom, then
\begin{equation*}
\big||\mathcal{E}(H,G)| - n^{m}p^{e(H)}\big| < \eps n^{m}p^{e(H)}.
\end{equation*}
\end{theorem}

This paper is organized as follows. In Section~\ref{sec:AuxiliaryResults} we state an important result, 
Lemma~\ref{lemma:Tuple2forTuplek}, and we give some results that are needed for the proof of 
Lemma~\ref{lemma:Tuple2forTuplek}.
In Section~\ref{sec:ExtensionAndCorollaries} we state the so-called ``Extension Lemma'', an important 
step in the proof of Theorem~\ref{thm:emb}. 
In Section~\ref{sec:proofs}, we prove Lemma~\ref{lemma:Tuple2forTuplek} and Theorem~\ref{thm:emb}. We finish with some concluding 
remarks in Section~\ref{sec:concluding}.

\section{Auxiliary results}\label{sec:AuxiliaryResults}

We begin by generalizing the definitions of $\BDD$ and $\TUPLE$ to deal not only with sets of $k-1$ vertices, but with sets of 
$i$ vertices, for any $1\leq i\leq k-1$.

\begin{property}[General Boundedness Property]
Let $k\geq 2$ and $1\leq i\leq k-1$. We define $\BDD_i(d,C,p)$ as the family of $n$-vertex 
$k$-uniform hypergraphs 
$G=(V,E)$ such that, for all $1\leq r\leq d$ and all families of distinct sets $S_1,\ldots,S_r\in 
{V\choose i}$, we have
 \begin{equation}  |N_G(S_1)\cap\ldots\cap N_G(S_r)|  \leq Cn^{k-i}p^r.
 \end{equation}
Note that $\BDD_{k-1}(d,C,p)$ is the same as $\BDD(d,C,p)$.
\end{property}

\begin{property}[General Tuple Property]
Let $k\geq 2$ and $1\leq i\leq k-1$. We define $\TUPLE_i(d,\delta,p)$ as the family of $n$-vertex 
$k$-uniform hypergraphs 
$G=(V,E)$ such that, for all $1\leq r\leq d$, the following holds.
\begin{equation}\label{gtuple-def}
\left||N_G(S_1)\cap\ldots\cap N_G(S_r)|-{n\choose{k-i}}p^r\right|<\delta {n\choose{k-i}}p^r
\end{equation}
for all but at most $\delta{{n\choose i}\choose r}$ families $\{S_1,\ldots,S_r\}$ of
$r$ distinct sets of ${V\choose i}$. 
We note that $\TUPLE_{k-1}(d,\delta,p)$ is the same as $\TUPLE(d,\delta,p)$.
\end{property}

Let $d\geq 2$ be an integer and let $\delta>0$. Roughly 
speaking, the next result (Lemma~\ref{lemma:Tuple2forTuplek}) states that if $G$ is a 
$(2,C,2,\delta',p)$-pseudorandom $k$-uniform 
hypergraph on $n$ vertices and $p=p(n)\gg n^{-1/d}$, then $G$ is in fact 
$(2,C,d,\delta,p)$-pseudorandom for all
sufficiently large $n$ as long as $\delta'$ is sufficiently small.

\begin{lemma}\label{lemma:Tuple2forTuplek}
For all $\delta>0$, $C>1$ and integers $k$, $d\geq 2$, there exists ${\delta'}>0$ such that the following holds when 
$p=p(n)\gg n^{-1/d}$ and $n$ is sufficiently large:
if $G$ is a $(2,C,2,\delta',p)$-pseudorandom $k$-uniform hypergraph, then $G$ 
is $(2,C,d,\delta,p)$-pseudorandom.
\end{lemma}

Since we have $n^{-1/D_H}\geq n^{-1/d_H}$ for any $k$-graph $H$, Lemma~\ref{lemma:Tuple2forTuplek} tells us that it 
suffices to consider
$(D_H,C,d_H,\delta,p)$-pseudorandom hypergraphs $G$ in the proof of Theorem~\ref{thm:emb}.

In the remainder of this section we prove some results that are important in the proof of 
Lemma~\ref{lemma:Tuple2forTuplek}. We start with some simple combinatorial facts and in 
Section~\ref{subsec:LmtCg} we, roughly speaking, show how to obtain properties $\BDD_i$ for every $1\leq i\leq k-1$ and 
$\TUPLE_1$ from our pseudorandomness assumption.
The proof of the following well-known lemma can be seen in~\cite{KoRoSi04}.

\begin{fact}\label{lemma:CStype}
 For every $\delta>0$, there exists $\gamma>0$ such that, if a family of real numbers $a_i\geq 0$, for
$1\leq i\leq N$, satisfies the following inequalities:
 \begin{enumerate}[label=\rmlabel]
  \item\label{it:csi} $\sum_{i=1}^N a_i \geq (1-\gamma)Na$,
  \item\label{it:csii} $\sum_{i=1}^N {a_i}^2 \leq (1+\gamma)Na^2$,
 \end{enumerate}
then
\begin{equation*}
\big|\{i\colon |a_i-a|<\delta a\}\big| > (1-\delta)N.
\end{equation*}
\end{fact}
  
Let $r\geq 1$ and $a>0$. Since ${na\choose r}/a^r{n\choose r}\to 1$ when $n\to\infty$, we 
obtain the following lemma, which provides a combinatorial inequality that will be used often. 

\begin{fact}\label{lemma:comb}
Let $\sigma>0$, $r\geq 1$ and $a>0$. Then, the following holds for a sufficiently large~$n$.
\begin{equation*}
 \left|{na\choose r} - a^r{n\choose r}\right| \leq \sigma a^r{n\choose r}.
\end{equation*}
\end{fact}

\subsection{Extending properties \texorpdfstring{$\BDD$}{BDD} and \texorpdfstring{$\TUPLE$}{TUPLE}}\label{subsec:LmtCg}
In this section we prove two results, 
Lemmas~\ref{lemma:degree-upper}~and~\ref{lemma:TUPLE-k-1--TUPLE-1}, that give conditions for a 
hypergraph $G$ to satisfy properties $\BDD_i$ for every $1\leq i\leq k-1$, and $\TUPLE_1$.

\begin{lemma}\label{lemma:degree-upper}
Let $C>1$ be an integer, let $G$ be an $n$ vertex $k$-uniform hypergraph and consider $0<p=p(n)\leq 1$. If $G$ 
satisfies $\BDD(2,C,p)$, then $G$ satisfies $\BDD_i(2,C,p)$ for all $1\leq 
i\leq k-1$.
\end{lemma}

\begin{proof}
The proof follows by induction on $i=k-1,\ldots,1$ and a simple averaging argument.
\end{proof}

The next result gives necessary conditions for a $k$-uniform hypergraph to satisfy 
property $\TUPLE_1(2,\delta,p)$.

\begin{lemma}\label{lemma:TUPLE-k-1--TUPLE-1}
For all $C>1$, ${\delta}>0$ and an integer $k\geq 2$, there exists $\sigma>0$ such that the 
following holds for $p\gg n^{-1/2}$ and sufficiently large $n$.

If $G$ is a $(2,C,2,\sigma,p)$-pseudorandom $n$-vertex $k$-uniform hypergraph, then $G$ satisfies 
$\TUPLE_1(2,\delta,p)$.
\end{lemma}

\begin{proof}
We must prove that \eqref{gtuple-def} holds for $1\leq r\leq 2$. 
Since the proofs of the cases $r=1$ and $r=2$ are similar, we present only the proof for the 
case $r=2$. 
We will show that the two inequalities required to apply Fact~\ref{lemma:CStype} hold.

Fix $C>1$, ${\delta}>0$ and an integer $k\geq 2$.
Let $\gamma>0$ be obtained by an application of Fact~\ref{lemma:CStype} with parameter 
${\delta}>0$ and let $\sigma=\sigma(C,\gamma)$ be a sufficiently small constant.
Now let $p\gg n^{-1/2}$ and consider a sufficiently large $n$. 
Suppose that $G=(V,E)$ is a $(2,C,2,\sigma,p)$-pseudorandom $n$-vertex $k$-uniform hypergraph. Thus,

\begin{align}\label{eq:inequality-part2-i}
\sum_{\{u,v\}\in {V\choose 2}} |N(u)\cap N(v)| &= \sum_{S\in {V\choose k-1}}{|N(S)|\choose 
2}\nonumber\\
&\geq \sum_{S\in {V\choose k-1}\colon |N(S)|\geq (1-{\sigma})np}{|N(S)|\choose 2}\nonumber\\
&\geq (1-{\sigma}){n\choose k-1}{(1-{\sigma})pn \choose 2}\nonumber\\
&\geq (1-{\sigma})^4{n\choose 2} {n\choose k-1}p^2\nonumber\\
&\geq (1-\gamma){n\choose 2} {n\choose k-1}p^2,
\end{align}
where the first inequality is trivial and the second follows from $\TUPLE(2,\sigma,p)$ and in the 
third inequality we 
apply Fact~\ref{lemma:comb}.

Heading for an application of Fact~\ref{lemma:CStype} we consider the following sum.
\begin{align}\label{eq:paraFinalizar2}
\sum_{\{u,v\}\in {V\choose 2}} |N(u)\cap N(v)|^2 &= \sum_{(S_1,S_2)\in {V\choose 
k-1}^2}{|N(S_1)\cap 
N(S_2)| \choose 2}\nonumber\\
&=\sum_{\substack{(S_1,S_2)\in {V\choose k-1}^2\\S_1\neq S_2}}{|N(S_1)\cap N(S_2)| \choose 2} + 
\sum_{S_1\in 
{V\choose k-1}}{|N(S_1)| \choose 
2}.
\end{align}
We will bound the two sums in~\eqref{eq:paraFinalizar2}. 
By~$\BDD(2,C,p)$, the choice of $p$ and an application of Fact~\ref{lemma:comb}, we have
\begin{align*}
\sum_{S_1\in {V\choose k-1}}{|N(S_1)| \choose 2} &\leq {n\choose k-1}{Cnp\choose 2}\nonumber\\ 
&\leq (1+{\sigma})C^2{n\choose 2}{n\choose k-1}p^2.
\end{align*}
Since $p\gg n^{-1/2}$, we obtain
\begin{align}\label{eq:part2-1}
\sum_{S_1\in {V\choose k-1}}{|N(S_1)| \choose 2}\leq{\sigma}{n\choose 2}\left({n\choose k-1}p^2\right)^2.
\end{align}

To bound the remaining sum, define $A$ and $B$ as the families of pairs $\{S_1,S_2\}$ with 
$S_1,S_2\in {V\choose k-1}$ and $S_1\neq S_2$ such that $|N(S_1)\cap N(S_2)|\leq 
(1+{\sigma})np^2$ for all pairs in $A$, and 
$|N(S_1)\cap N(S_2)|> (1+{\sigma})np^2$ for all pairs in $B$.
By Fact~\ref{lemma:comb} we obtain
\begin{align}\label{eq:part2-2}
\sum_{\{S_1,S_2\}\in A}{|N(S_1)\cap N(S_2)| \choose 2}&\leq \frac{1}{2}{n\choose 
k-1}^2{(1+{\sigma})p^2{n}\choose 2}\nonumber\\
&\leq \frac{(1+{\sigma})^3}{2}{n\choose 2}\left({n\choose k-1}p^2\right)^2.
\end{align}
By $\BDD(2,C,p)$, $\TUPLE(2,\sigma,p)$ and Fact~\ref{lemma:comb} applied with
$\sigma$, $r=2$ and $a=Cp^2$, we have
\begin{align}\label{eq:part2-3}
 \sum_{\{S_1,S_2\}\in B}{|N(S_1)\cap N(S_2)| \choose 2} &\leq \frac{{\sigma}}{2} {n\choose k-1}^2 
{Cp^2 n\choose 2}\nonumber\\
&\leq \frac{{\sigma}(1+{\sigma}) C^2}{2}{n\choose 2}\left({n\choose k-1}p^2\right)^2.
\end{align}
Replacing \eqref{eq:part2-1}, \eqref{eq:part2-2} and \eqref{eq:part2-3} in 
\eqref{eq:paraFinalizar2}, we have
\begin{align}\label{eq:inequality-part2-ii}
\sum_{\{u,v\}\in {V\choose 2}} |N(u)\cap N(v)|^2 &\leq \left({\sigma} + (1+{\sigma})^{3} + 
{\sigma}(1+{\sigma})C^2\right){n\choose 
2}\left({n\choose 
k-1}p^2\right)^2\nonumber\\
&\leq (1+\gamma){n\choose 2}\left({n\choose k-1}p^2\right)^2.
\end{align}

Equations \eqref{eq:inequality-part2-i} and \eqref{eq:inequality-part2-ii} can be seen as  
inequalities~\ref{it:csi} and~\ref{it:csii} in Fact~\ref{lemma:CStype}. 
Therefore, we conclude that, for at least $(1-{\delta}){n\choose 2}$ pairs of vertices $\{u,v\} 
\in 
{V\choose 2}$, we have
\begin{equation*}
\left||N(u)\cap N(v)|-{n\choose k-1}p^2\right|<{\delta} {n\choose k-1}p^2.
\end{equation*}
\end{proof}

\section{Extension Lemma and corollaries}\label{sec:ExtensionAndCorollaries}

In this section we prove a result called \emph{Extension Lemma} (Lemma~\ref{lemma:ExtensionLemma}) 
from where we derive 
Corollaries~\ref{cor:NIemb}~and~\ref{cor:POLLemb}, which are used in the proof of 
Theorem~\ref{thm:emb}.

\subsection{Extension Lemma}

Before starting the discussion concerning the Extension Lemma we shall define some concepts. 
Consider $k$-uniform hypergraphs $G$ and $H$. Given sequences $W=w_1,\ldots,w_\ell\in 
{V(H)}^\ell$ and $X=x_1,\ldots,x_\ell\in {V(G)}^\ell$, 
define $\mathcal{E}(H,G,W,X)$ as the set of embeddings $f\in \mathcal{E}(H,G)$ such that 
$f(w_i)=x_i$ for all $1\leq i\leq \ell$. 
Furthermore, for a sequence $Y$ define the set of its elements by $Y^{\set}=\{y_1,\ldots,y_\ell\}$. We say that a subset of vertices $V'\subset 
V(H)$ is \emph{stable} if 
$E(H[V'])=\emptyset$, i.e., if there is no edge of $H$ contained in~$V'$.

Let $H$ be a hypergraph with $m$ vertices. 
We say that $H$ is $d$\emph{-degenerate} if there exists an ordering $v_1,\ldots,v_m$ of $V(H)$ 
such that 
$d_{H_i}(v_i)\leq d$ for all $1\leq i\leq m$, where $H_i=H[\{v_1,\ldots,v_i\}]$. In this case, 
we say that $v_1,\ldots,v_m$ is a 
\emph{$d$-degenerate ordering} of the vertices of $H$. Given a sequence $W\in V(H)^{\ell}$, we 
define 
$\omega(H,W)=|{E(H)}| 
- |E(H[W^{\set}])|$, i.e., $\omega(H,W)$ is the number of edges of $H$ that are not contained in 
$W^{\set}$.

\begin{lemma}[Extension Lemma]\label{lemma:ExtensionLemma}
Let $C\geq 1$, $m\geq 1$ and $k\geq 2$. Let $G$ and $H$ be $k$-uniform hypergraphs such that $H$ is 
linear, $|V(H)|=m$, 
$|V(G)|=n$ and $p=p(n)=e(G)/{n\choose k}$. Suppose that $0\leq 
\ell\leq \max\{k,d_H\}$, and let $W\in {V(H)}^\ell$ and $X\in {V(G)}^\ell$ be fixed. If $G\in 
\BDD(D_H,C,p)$, then
\begin{equation*}
|\mathcal{E}(H,G,W,X)| \leq C^{m-\ell}n^{m-\ell}p^{\omega(H,W)}.
\end{equation*}
In particular, if $W^{\set}\subset V(H)$ is stable, then $|\mathcal{E}(H,G,W,X)| \leq 
C^{m-\ell}n^{m-\ell}p^{e(H)}$.
\end{lemma}

For a $k$-uniform hypergraph $H=(V,E)$ with $|V|=m$ vertices and for a positive integer~$\ell\leq \max\{k,d_H\}$, 
Proposition~\ref{lemma:creatingOrdering} allows us to obtain a $D_H$-degenerate ordering 
$v_1,\ldots, v_m$ of~$V$ such that 
$W=v_1,\ldots,v_\ell$ from a $d_H$-degenerate ordering of $V$. Consider a sequence
$L$ of the vertices of $H$. Given a subsequence 
$W$ of $V$, we write $L\setminus W$ for the sequence of $L\setminus W$ obtained from $L$ 
by deleting
the vertices of $W$. Given a sequence of 
vertices $Y$ in $V^{\ell}$, 
we write $L'=(Y,L\setminus Y)$ to denote the 
sequence $L'$ of 
$V$ obtained by removing $Y$ from $L$ and placing it before the elements of $L$.

\begin{proposition}\label{lemma:creatingOrdering}
Let $H=(V,E)$ be a linear $k$-uniform hypergraph and  let $\ell$ be an integer with $0\leq \ell\leq \max\{k,d_H\}$.  
If $W\in {V}^\ell$, then there exists a $D_H$-degenerate 
ordering $w_1,\ldots,w_{|V|}$ of $V$ such that
$W=w_1,\ldots,w_\ell$.
\end{proposition}

\begin{proof}
Fix $k\geq 2$ and let $H$, $\ell$ and $W$ as in the statement of the proposition. Note that the result is trivial 
whenever $W$ is empty, and if $D_H=\Delta(H)$, then any ordering 
of the vertices of $H$ is $D_H$-degenerate. Therefore, assume $D_H = k\cdot d_H$ and $1\leq 
|W|=\ell\leq \max\{k,d_H\}$.

Let $L$ be a $d_H$-degenerate ordering of $V$ and put 
$L'=(W,L\setminus W$). Given a 
vertex 
$v$ of~$H$, 
define the \emph{left degree} of $v$ in $L'$ as the number of edges $e$ such that $v$ is the 
rightmost element of  
$e$ considering the ordering $L'$. Since $L$ is $d_H$-degenerate  and, by the linearity of $H$, any vertex $v$ belongs to at most
$|W|$ edges containing  vertices of $W$, the left degree of $v$ in 
$L'$ 
is at most 
$|W|+d_H$. We divide the proof into three cases.\\
\textbf{Case 1:} $d_H>k$. In this case, $|W|\leq d_H$. Then, the left degree of any 
vertex 
of $H$ in 
$L'$ is at 
most $2d_H \leq k\cdot d_H=D_H$. Therefore, $L'$ is a $D_H$-degenerate ordering of $V$.\\
\textbf{Case 2:} $2\leq d_H\leq k$. Here we have $|W|\leq k$. Therefore, since $k\geq 
2$, 
the left degree of 
each vertex of 
$H$ in $L'$ is at most $k+d_H \leq k\cdot d_H=D_H$. Therefore, $L'$ 
is a $D_H$-degenerate ordering of $V$.\\
\textbf{Case 3:} $d_H=1$. Here we have $|W|\leq k$. Note that 
the only possibility for a vertex 
$v$ to have left degree larger than $k\cdot d_H=k$ in $L'$ is if the following holds: $|W|=k$ and, for every $w\in W$, the vertex 
$v$ belongs to an 
edge $e_w$ containing $w$ and $w$ is the rightmost element of~$e_w$ in $L$. But note that, 
since 
$d_H=1$, there exists at 
most one vertex $v$ with this property, otherwise $L$ would not be a $d_H$-degenerate ordering. 
Let $W'$ be the ordering $w_1,\ldots,w_\ell,v$. Now consider the 
ordering 
$L''=(W',L\setminus W')$. It is clear that all the vertices of $H$ have left 
degree at most 
$2\leq k\cdot d_H=D_H$ in $L''$. Therefore, $L''$ is a $D_H$-degenerate ordering of $V$.
\end{proof}

Now we prove the Extension Lemma.

\begin{proof}[Proof of Lemma~\ref{lemma:ExtensionLemma}]
Fix $C\geq 1$, $m\geq 1$ and $k\geq 2$. Let $G$ and $H$ be $k$-uniform hypergraphs such that $H$ 
is linear with $|V(H)|=m$, $|V(G)|=n$ and 
$p=p(n)=e(G)/{n\choose k}$.
Let $\ell$ be an integer with $0\leq \ell\leq \max\{k,d_H\}$, and let $W\in {V(H)}^\ell$ and $X\in {V(G)}^\ell$. Suppose that 
$G\in \BDD(D_H,C,p)$.
By Proposition~\ref{lemma:creatingOrdering}, we know that there exists 
a $D_H$-degenerate ordering $v_1,\ldots,v_m$ of $V(H)$ such that $W$ is its initial segment. 
We will prove by induction on $h$ that, for all $\ell\leq h\leq m$, 
\begin{equation}\label{eq:AimClaim}
|\mathcal{E}(H_h,G,W,X)| \leq C^{h-\ell}n^{h-\ell}p^{\omega(H_h,W)},
\end{equation}
where $H_h=H[\{v_1,\ldots,v_h\}]$. 

If $h=\ell$, the statement is trivial. Suppose that 
$\ell<h\leq m$ 
and
\begin{equation*}
|\mathcal{E}(H_{h-1},G,W,X)| \leq C^{{h-1}-\ell}n^{{h-1}-\ell}p^{\omega(H_{h-1},W)}.
\end{equation*}
Since $v_1,\ldots, v_m$ is $D_H$-degenerate we have $d_{H_h}(v_h)\leq 
D_H$. By $G\in \BDD(D_H,C,p)$, we know that any embedding from $H_{h-1}$ 
to $G$ can be extended to an embedding from $H_h$ to $G$ in at most
$Cnp^{d_{H_h}(v_h)}$ different ways. Since
$\omega(H_h,W) = \omega(H_{h-1},W) + d_{H_h}(v_h)$, applying the induction hypothesis, we conclude 
that
\begin{equation*}
\begin{aligned}
|\mathcal{E}(H_h,G,W,X)| &\leq Cnp^{d_{H_h}(v_h)}|\mathcal{E}(H_{h-1},G,W,X)|\\
&\leq Cnp^{d_{H_h}(v_h)}C^{h-1-\ell}n^{h-1-\ell}p^{\omega(H_{h-1},W)}\\
&= C^{h-\ell}n^{h-\ell}p^{\omega(H_h,W)}.
\end{aligned}
\end{equation*}
\end{proof}

\subsection{Corollaries of the Extension Lemma}

Given $k$-uniform hypergraphs $G$ and $H$, we write $\mathcal{E}^{\ni}(H,G)$ and 
$\mathcal{E}^{\ind}(H,G)$ for the set 
of non-induced and induced embeddings 
from $H$ into $G$, respectively. 
The following corollary bounds from above the number of embeddings in $\mathcal{E}^{\ni}(H,G)$ for 
some hypergraphs $G$ 
whenever $H$ is linear.

\begin{corollary}\label{cor:NIemb}
Let $C\geq 1$, $m,k,\eta>0$ and $p=p(n)=o(1)$ with $m\geq k\geq 2$. Then,  
for all $k$-uniform hypergraphs $G$ and $H$, where $|V(G)|=n$ and $H$
is linear with $|V(H)|=m$ the following holds. If 
$G\in 
\BDD(D_H,C,p)$ and $n$ is sufficiently large, then 
\begin{equation*}
\big|\mathcal{E}^{\ni}(H,G)\big|<\eta n^mp^{e(H)}. 
\end{equation*}
\end{corollary}

\begin{proof}
Fix $C\geq 1$, $m\geq k\geq 2$, $\eta>0$ and let $p=p(n)=o(1)$. Let $G$ and $H$ be as in the 
statement and let $n$ be sufficiently large. 

Fix an edge~$\{x_1,\ldots,x_k\}\in E(G)$ and a non-edge $\{w_1,\ldots,w_k\}$ of $H$. Applying
Lemma~\ref{lemma:ExtensionLemma} 
with
$W=(w_1,\ldots,w_k)$ and $X=(x_1,\ldots,x_k)$, we conclude that the number of embeddings~$f$ from 
$V(H)$ into $V(G)$ 
such 
that $f(w_i)=x_i$ for $1\leq i\leq k$ is bounded from above by $C^{m-k}n^{m-k}p^{{E(H)}}$. Since $G\in \BDD(D_H,C,p)$, we 
have $|E(G)|\leq Cn^kp$, from where we conclude that there exist at 
most $Cn^kp$ choices for $\{x_1,\ldots,x_k\}$ in $E(G)$. Note that there exist at most ${m\choose k}$ choices for 
$\{w_1,\ldots,w_k\}$ 
in~${V(H)\choose k}$. Then, we can choose $(x_1,\ldots,x_k)$ and $(w_1,\ldots,w_k)$, 
respectively, in $Ck!n^kp$ and $k!{m\choose k}$ ways. Therefore, 
$|\mathcal{E}^{\ni}(H,G)|\leq K n^{m}p^{e(H)+1}$ for some constant $K=K(C,k,m)$. Since $p=o(1)$ the lemma follows for any 
$\eta>0$ and any sufficiently large $n$.
\end{proof}

Let $G$ and $H$ be $k$-uniform hypergraphs with $|V(G)|=n$ and consider a set \mbox{$X\subset {V(H)\choose 
k-1}$}. If $f$ 
is an 
embedding from $H$ into $G$, we denote by
$f_{k-1}(X)$ the family $\{f(x_1),\ldots,f(x_{k-1})\}$, for all
$\{x_1,\ldots,x_{k-1}\}\in X$.

Given $\delta>0$, define $B_G(\delta,r)$ as the families $\{X_1,\ldots,X_r\}$ of $r$ distinct sets of ${V(G)\choose k-1}$ such 
that
\begin{equation*}
\big||N_{G}(X_1)\cap\ldots\cap N_G(X_r)|-np^r\big|\geq \delta np^r.
\end{equation*}
Consider the following definition.
\begin{equation*}
B_G^{\stb}(\delta,r)=\left\{\{X_1,\ldots,X_r\}\in B_G(\delta,r)\colon \bigcup_{i=1}^r X_i\text{ is stable in }G\right\}.
\end{equation*}
Given $r$ distinct sets $X_1,\ldots,X_r$ of ${V(G)\choose k-1}$, we say that 
$X=\{X_1,\ldots,X_r\}$ is $\delta$-\emph{bad} if $X\in B_G^{\stb}(\delta,r)$.
Let $H$ be a $k$-uniform hypergraph with $m$ vertices and let $v_1,\ldots, v_m$ be a 
$d_H$-degenerate ordering 
of $V(H)$. Define $H_i=H[v_1,\ldots,v_i]$. 
We say that an embedding $f\colon V(H_{h-1})\to V(G)$ is \text{$\delta$-\emph{clean}} if   
$f_{k-1}(N_{H_h}(v_h))\notin B_G^{\stb}(\delta,d_{H_h}(v_h))$. Moreover, if $f\colon V(H_{h-1})\to 
V(G)$ is not $\delta$-clean, then we say that $f$ 
is $\delta$-\emph{polluted}. We denote the set of embeddings $f\in 
\mathcal{E}(H_{h-1},G)$ such that $f$ is $\delta$-polluted by 
$\mathcal{E}_{\poll}(H_{h-1},G)$. Similarly, 
we denote by $\mathcal{E}_{\clean}(H_{h-1},G)$ the set of embeddings $f\in 
\mathcal{E}(H_{h-1},G)$ such that 
$f$ is $\delta$-clean.
The next corollary shows that if $H$ is linear and connector-free then most of the embeddings from $H_{h-1}$ into a 
sufficiently pseudorandom hypergraph $G$ are clean, for $1\leq h\leq m$.

\begin{corollary}\label{cor:POLLemb}
Let $\delta>0$, $C>1$, $m\geq 4$ and $k\geq 2$ be fixed constants. Let $H$ be an $m$-vertex linear 
$k$-uniform hypergraph that is connector-free and let
$v_1,\ldots, v_m$ be a $d_H$-degenerate ordering of~$V(H)$. Suppose that $1<h\leq m$ and put
$r=d_{H_h}(v_h)$. If $G$ is 
$(D_H,C,d_H,\delta,p)$-pseudorandom, then
\begin{equation*}
 |\mathcal{E}_{\poll}(H_{h-1},G)|\leq 
\delta\left(r!((k-1)!)^rC^{h-1-r(k-1)}\right)n^{h-1}p^{e(H_{h-1})}.
\end{equation*}
\end{corollary}

\begin{proof}
Fix constants $\delta>0$, $C>1$, $m\geq 4$ and $k\geq 2$. Let $H$ be an $m$-vertex linear 
$k$-uniform hypergraph that 
is 
connector-free. Consider a $d_H$-degenerate ordering $v_1,\ldots, v_m$ of $V(H)$. Let $1<h\leq m$ 
and put 
$r=d_{H_h}(v_h)$. Suppose that $G$ is $(D_H,C,d_H,\delta,p)$-pseudorandom.

By definition, an embedding $f\colon V(H_{h-1})\to V(G)$ is $\delta$-polluted if 
$f_{k-1}(N_{H_h}(v_h))\in B_G^{\stb}(\delta,r)$. Let $N_{H_h}(v_h)=\{W_1,\ldots,W_r\}$ where
$W_i=\{w_{i,1},\ldots,w_{i,{k-1}}\}$ for all $1\leq i\leq r$ (Note that since $H$ is linear, the sets $W_1,\ldots,W_r$ are 
pairwise disjoint). Let 
\[
	W_{\ord}=(w_{1,1},\ldots,w_{1,{k-1}},w_{2,1},\ldots,w_{2,{k-1}},\ldots,w_{r,1}\ldots,w_{r,{k-1}})
\]
be an ordering of $W_1\cup\ldots\cup W_r$. Therefore, 
\begin{equation*}
\mathcal{E}_{\poll}(H_{h-1},G) = \bigcup_{X}\left(\bigcup_{X_{\ord}} 
\mathcal{E}(H_{h-1},G,W_{\ord},X_{\ord})\right),
\end{equation*}
where the first union is over all families ${X=\{S_1,\ldots,S_r\}\in 
B_G^{\stb}(\delta,r)}$ and the second union is 
over all $((k-1)!)^r$ possible orderings of $S_i$ for $1\leq i\leq r$, and all $r!$ orderings of $X$.
Therefore, 
\begin{equation*}
|\mathcal{E}_{\poll}(H_{h-1},G)| \leq  \sum_{X}\sum_{X_{\ord}} 
|\mathcal{E}(H_{h-1},G,W_{\ord},X_{\ord})|.
\end{equation*}
Note that, since $H_h$ is linear and connector-free, $\bigcup N_{H_h}(v_h)$ is stable in $H_h$. 
Since $G\in \BDD(D_H,C,p)$ and $|W_{\ord}|=r(k-1)$, we know from 
the conclusion of Lemma~\ref{lemma:ExtensionLemma} that
\begin{equation*}
|\mathcal{E}(H_{h-1},G,W_{\ord},X_{\ord})|\leq C^{h-1-r(k-1)}n^{h-1-r(k-1)}p^{e(H_{h-1})}.
\end{equation*}
Since ${r=d_{H_h}(v_h)\leq 
d_H}$ and $G$ satisfies $\TUPLE(d_H,\delta,p)$, we have
$\big|B_G^{\stb}(\delta,r)\big|\leq \delta n^{r(k-1)}$. Then, the first of the sums contains at 
most $\delta n^{r(k-1)}$ terms. Since the 
second sum is over $r!((k-1)!)^r$ terms, we obtain
\begin{equation*}
|\mathcal{E}_{\poll}(H_{h-1},G)| \leq \delta\left(r!((k-1)!)^rC^{h-1-r(k-1)} 
\right)n^{h-1}p^{e(H_{h-1})}.
\end{equation*}
\end{proof}

\section{Proof of the main result}\label{sec:proofs}

Before proving Theorem~\ref{thm:emb} we prove Lemma~\ref{lemma:Tuple2forTuplek}. The proof of 
Lemma~\ref{lemma:Tuple2forTuplek} is simple and rely on Facts~\ref{lemma:CStype}~and~\ref{lemma:comb}, and 
Lemma~\ref{lemma:degree-upper}. For simplicity, we will not explicit the constants used in its proof.

\begin{proof}[Proof of Lemma~\ref{lemma:Tuple2forTuplek}]
Fix $\delta>0$, $C>1$ and integers $k,d\geq 2$, and let $2\leq 
r\leq d$. Let $\gamma>0$ be obtained by an application of Fact~\ref{lemma:CStype} with parameters 
$\delta$. Now let $\sigma=\sigma(k,r,\gamma)$ be a sufficiently small constant.
Let ${\delta_{\ref{lemma:TUPLE-k-1--TUPLE-1}}}$ be obtained by an application of 
Lemma~\ref{lemma:TUPLE-k-1--TUPLE-1} with 
parameter $C$, ${\sigma}$ and $k$ and put 
$\delta'=\min\{\delta,\delta_{\ref{lemma:TUPLE-k-1--TUPLE-1}}\}$. Consider $p\gg n^{-1/d}$ and let $n$ be sufficiently 
large.

Suppose $G=(V,E)$ is an $n$-vertex 
$k$-uniform $(2,C,2,\delta',p)$-pseudorandom hypergraph. By Lemma~\ref{lemma:TUPLE-k-1--TUPLE-1}, 
the following two inequalities hold, respectively, for more than $(1-{\sigma})n$ vertices $u \in 
V$ and for more than $(1-{\sigma}){n\choose 2}$ pairs $\{u,v\} \in 
{V\choose 2}$.
\begin{align}
\left||N(u)|-{n\choose k-1}p\right|&<{\sigma} {n\choose k-1}p, \label{eq:LemmaBDD-i}\\
\left||N(u)\cap N(v)|-{n\choose k-1}p^2\right|&<{\sigma} {n\choose k-1}p^2.\label{eq:LemmaBDD-ii}
\end{align}
We must check that the inequalities~\ref{it:csi} and~\ref{it:csii} of Fact~\ref{lemma:CStype} hold.
For inequality~\ref{it:csi}, consider the following sum over distinct sets $S_1, \ldots, S_r\in {V\choose k-1}$.
\begin{align}\label{eq:inequality-for-CS-i}
\sum_{S_1, \ldots, S_r\in {V\choose k-1}} 
|N(S_1)\cap\ldots\cap N(S_r)| &= 
\sum_{u\in V}{|N(u)|\choose r}\nonumber\\
&\geq (1-{\sigma})n{(1-\sigma){n\choose k-1}p\choose r}\nonumber\\
&\geq (1-\gamma){{n\choose k-1}\choose r}np^r,
\end{align}
where the first inequality follows from \eqref{eq:LemmaBDD-i} and the last one follows from Fact~\ref{lemma:comb}.
It remains to prove that inequality~\ref{it:csii} of Fact~\ref{lemma:CStype} holds. Consider the following sum over distinct sets $S_1, 
\ldots, S_r\in {V\choose k-1}$.
\begin{align}\label{eq:limitaIntersecao}
\sum_{S_1, \ldots, S_r\in {V\choose k-1}}
\left|\bigcap_{i=1}^r N(S_i)\right|^2 
&= \sum_{(u,v)\in V^2}{|N(u)\cap N(v)| \choose r}\nonumber\\
&=\sum_{\substack{(u,v)\in V^2\\ u\neq v}}{|N(u)\cap N(v)| \choose r} + \sum_{u\in V}{|N(u)| \choose 
r}.
\end{align}
Let us estimate the sums in~\eqref{eq:limitaIntersecao}. In view of Lemma~\ref{lemma:degree-upper} applied for $i=1$ we can apply 
the boundedness property to bound $|N(u)|$ for every $u\in V$, obtaining
\begin{equation}\label{eq:final1}
\sum_{u\in V}{|N(u)| \choose r} \leq n{Cn^{k-1}p\choose r}\leq C' {{n\choose k-1} \choose r} np^r
\end{equation}
for some $C'=C'(k,r,\sigma)$.

Now we estimate the remaining sum. Define $A$ and $B$ as the families of pairs $\{u,v\}\in 
{V\choose 2}$ such that $|N(u)\cap N(v)|\leq (1+{\sigma} ){n\choose k-1}p^2$ and $|N(u)\cap 
N(v)|> (1+{\sigma} ){n\choose k-1}p^2$, respectively. Since $p^2n^{k-1}\gg 1$, Fact~\ref{lemma:comb} implies
\begin{equation}\label{eq:final2}
\sum_{\{u,v\}\in A}{|N(u)\cap N(v)| \choose r}\leq \frac{n^2}{2}{(1+{\sigma} )p^2{n\choose 
k-1}\choose r}
\leq \frac{(1+{\sigma'})}{2}{{n\choose k-1}\choose r}(np^r)^2,
\end{equation}
where $\sigma'=\sigma'(r,\sigma)$ is a sufficiently small constant.
Similarly, using the boundedness of $G$ and \eqref{eq:LemmaBDD-ii} we obtain
\begin{equation}\label{eq:final3}
\sum_{\{u,v\}\in B}{|N(u)\cap N(v)| \choose r} \leq \frac{{\sigma n^2}}{2} 
{C(k-1)^{k-1}p^2{n\choose k-1}\choose r}
\leq \sigma'{{n\choose k-1}\choose r}(np^r)^2.
\end{equation}
Replacing \eqref{eq:final1}, \eqref{eq:final2} and \eqref{eq:final3} in 
\eqref{eq:limitaIntersecao}, we have
\begin{align}\label{eq:inequality-for-CS-ii}
\sum_{S_1, \ldots, S_r\in {V\choose k-1}} \left|\bigcap_{i=1}^r N(S_i)\right|^2 
\leq (1+\gamma){{n\choose k-1}\choose r}(np^r)^2,
\end{align}
where the above sum is over distinct sets $S_1, \ldots, S_r$.

Inequalities \eqref{eq:inequality-for-CS-i} and \eqref{eq:inequality-for-CS-ii} can be seen as 
inequalities~\ref{it:csi} and~\ref{it:csii} in Fact~\ref{lemma:CStype}. 
Therefore, we conclude that, for more than $(1-\delta){{n\choose k-1}\choose r}$ families of 
distinct sets
$S_1,\ldots,S_r \in {V\choose k-1}$, the following holds for all $1\leq r\leq d$.
\begin{equation*}
\big||N(S_1)\cap\ldots\cap N(S_r)|-np^r\big|<\delta np^r.
\end{equation*}
To finish the proof, note that, since $\delta'\leq\delta$ and $G\in \TUPLE(2,\delta',p)$, the 
following holds for more than $(1-\delta){n\choose 
k-1}$ sets
$S_1 \in {V\choose k-1}$.
\begin{equation*}
\big||N(S_1)|-np\big|<\delta np.
\end{equation*}
\end{proof}

\begin{proof}[Proof of Theorem~\ref{thm:emb}]

Let $k\geq 2$ and $m\geq 4$ be integers and fix $C>1$. Let $H$ be a linear $k$-uniform 
connector-free hypergraph on $m$ vertices. Fix a $d_H$-degenerate ordering $v_1,\ldots,v_m$
of $V(H)$ and put $H_h=H[\{v_1,\ldots,v_h\}]$. 

We will use induction on $h$ to prove that for every $1\leq h\leq m$ and for every $\eps>0$, there 
exists 
$\delta>0$ such that the following holds when $p\gg n^{-1/D_H}$ and $n$ is sufficiently large: if $G$ is an $n$-vertex 
$k$-uniform 
$(D_H,C,d_H,\delta,p)$-pseudorandom hypergraph
(recall that Lemma~\ref{lemma:Tuple2forTuplek} allows us to 
consider this stronger pseudorandomness condition on $G$), then
\begin{equation}\label{eq:AimLemma33}
\left||\mathcal{E}(H_h,G)| - n^hp^{e(H_h)}\right|<\varepsilon n^hp^{e(H_h)}.
\end{equation}

For every $\eps>0$ and $h=1$ the result is trivial. Thus, assume $1<h\leq m$ and suppose the result 
holds for 
$h-1$ and for all $\eps>0$. 

Let $\eps>0$ be given, let $\varepsilon'=\min\{\varepsilon/4,\varepsilon/6C\}$ and 
consider $\delta'=\delta'(\varepsilon')$ given by the induction hypothesis such 
that for $p\gg n^{-1/D_H}$  with $p=o(1)$ the 
following holds for sufficiently large~$n$.
\begin{equation}\label{eq:EmbeddingHh-1}
\left||\mathcal{E}(H_{h-1},G)| - n^{h-1}p^{e(H_{h-1})}\right|<\varepsilon' n^{h-1}p^{e(H_{h-1})}.
\end{equation}
Fix $\eta=\varepsilon'/2$ and define $r=d_{H_v}(v_h)\leq d_H$. Let
$\delta$ be a sufficiently small constant and 
suppose $p\gg n^{-1/D_H}$ with $p=o(1)$ and $n$ is sufficiently large.

Suppose $G$ is an $n$-vertex $k$-uniform 
$(D_H,C,d_H,\delta,p)$-pseudorandom hypergraph. An application of Corollary~\ref{cor:NIemb} with parameters 
$C$, 
$h-1$, $k$, $\eta$ and $p$ for the graphs $H_{h-1}$ and~$G$ provides the following upper bound on 
the number of 
non-induced embeddings.
\begin{equation}\label{eq:EmbeddingNI}
\left|\mathcal{E}^{\ni}(H_{h-1},G)\right|\leq\eta n^{h-1}p^{e(H_{h-1})}.
\end{equation}
By Corollary~\ref{cor:POLLemb} applied with $\delta$, $C$, $m$ and 
$k$ for the graphs 
$H_{h-1}$ and $G$, we have
\begin{equation}\label{eq:EmbeddingPOLL}
|\mathcal{E}_{\poll}(H_{h-1},G)|\leq\eta n^{h-1}p^{e(H_{h-1})}.
\end{equation}
By \eqref{eq:EmbeddingNI} and \eqref{eq:EmbeddingPOLL},
\begin{equation*}
\left|\mathcal{E}^{\ni}(H_{h-1},G)\cup \mathcal{E}_{\poll}(H_{h-1},G)\right|\leq 2\eta 
n^{h-1}p^{e(H_{h-1})} = 
\varepsilon'n^{h-1}p^{e(H_{h})-r}.
\end{equation*}
Then, \eqref{eq:EmbeddingHh-1} implies
\begin{equation}\label{eq:EmbeddingIndClean}
(1-2\varepsilon')n^{h-1}p^{e(H_{h})-r}<\left|\mathcal{E}^{\ind}_{\clean}(H_{h-1},
G)\right|<(1+\varepsilon')n^{h-1}p^{e(H_{h})-r}.
\end{equation}

The next step is to bound from below the number of ways we can extend an embedding $f'\in 
\mathcal{E}^{\ind}_{\clean}(H_{h-1},G)$ 
to an embedding $f\in \mathcal{E}(H_{h},G)$. Let $f'$ be such an embedding. Since~$f'$ is clean, 
$f'_{k-1}(N_{H_h}(v_h))\notin 
B_G^{\stb}(\delta,r)$, i.e., either $f'(\bigcup N_{H_h}(v_h))$ is not stable
in $G$ or $\big|N_{G}\big(f'_{k-1}(N_{H_h}(v_h))\big) - np^r\big|<\delta np^r$.
Since $H$ is linear and connector-free, it is easy to see that $\bigcup N_{H_h}(v_h)$ is stable in 
$H_h$. 
But since $f'$ is an induced embedding, 
$f'(\bigcup N_{H_h}(v_h))$ is stable in $G$. Therefore,
\begin{equation}\label{eq:RightNeig}
\big|N_{G}\big(f'_{k-1}(N_{H_h}(v_h))\big) - np^r\big|<\delta np^r.
\end{equation}

To obtain an extension $f\in\mathcal{E}(H_h,G)$ from $f'\in\mathcal{E}(H_{h-1},G)$ we must choose 
$f(v_h)$ in the set 
$N_{G}\big(f'_{k-1}(N_{H_h}(v_h))\big)\setminus f'\big(V(H_{h-1})\big)$. Therefore, the number of 
such extensions is 
\begin{equation}\label{eq:limitandoTaisExtensoes}
\big|N_{G}\big(f'_{k-1}(N_{H_h}(v_h))\big)\setminus f'\big(V(H_{h-1})\big)\big| \geq (1-\delta)np^r 
- (h-1)\geq 
(1-2\delta)np^r,
\end{equation}
where the first inequality is due to~\eqref{eq:RightNeig} and the last one follows from the choice 
of $p$. 
By~\eqref{eq:EmbeddingIndClean} and 
\eqref{eq:limitandoTaisExtensoes}, we have
\begin{equation*}
 \begin{aligned}
  |\mathcal{E}(H_h,G)| &\geq |\mathcal{E}^{\ind}_{\clean}(H_h,G)|\\
  &\geq 
\left|\mathcal{E}^{\ind}_{\clean}(H_{h-1},G)\right|\big|N_{G}\big(f'_{k-1}(N_{H_h}
(v_h))\big)\setminus f'\big(V(H_{h-1})\big)\big|\\
  &>(1-2\varepsilon')(1-2\delta)n^{h-1}p^{e(H_{h})-r}np^r\\
  &\geq (1-\varepsilon)n^hp^{e(H_h)}.
 \end{aligned}
\end{equation*}

To finish the proof we must show that $|\mathcal{E}(H_h,G)|<(1+\varepsilon)n^hp^{e(H_h)}$. Fix an 
embedding 
${f'\in \mathcal{E}(H_{h-1},G)}$. Consider the case 
$f'\in\mathcal{E}^{\ind}_{\clean}(H_{h-1},G)$. Note that the number of extensions of $f'$ to 
embeddings 
from $H_h$ into $G$ is at most $\big|N_{G}\big(f'_{k-1}(N_{H_h}(v_h))\big)\big|$. Therefore, 
by~\eqref{eq:EmbeddingIndClean} and~\eqref{eq:RightNeig}, the number of such embeddings is at most 
\begin{align}\label{eq:upperCleanInd}
\left|\mathcal{E}^{\ind}_{\clean}(H_{h-1},G)\right|\left|N_{G}\big(f'_{k-1}(N_{H_h}
(v_h))\big)\right| & \leq 
(1+\varepsilon')n^{h-1}p^{e(H_h)-r}(1+\delta)np^r\nonumber\\
&\leq (1+\varepsilon/2)n^hp^{e(H_h)}.
\end{align}

Now suppose ${f'\in\left\{\mathcal{E}(H_{h-1},G) \setminus 
\mathcal{E}^{\ind}_{\clean}(H_{h-1},G)\right\}}$. 
By~\eqref{eq:EmbeddingHh-1}~and~\eqref{eq:EmbeddingIndClean}, we have
\begin{equation}\label{eq:limitaCimaPolOuNaoInd}
\left|\mathcal{E}(H_{h-1},G) \setminus \mathcal{E}^{\ind}_{\clean}(H_{h-1},G)\right|\leq 
3\varepsilon'n^{h-1}p^{e(H_h)-r}.
\end{equation}
But since $r=d_{H_h}(v_h)\leq d_H\leq D_H$ and $G\in \BDD(D_H,C,p)$, every embedding $f'$ from~$H_{h-1}$ into $G$ can be extended to at most $Cnp^r$ 
embeddings $f\in \mathcal{E}(H_h,G)$. In fact, to see this, apply property $\BDD(D_H,C,p)$ to the 
family $\big\{f'(S_1),\ldots,f'(S_{|N_{H_h}(v_h)|})\big\}$,
where $\{S_1,S_2,\ldots,S_{|N_{H_h}(v_h)|}\}$ is the neighbourhood of $v_h$ in $H_h$. This fact together 
with~\eqref{eq:limitaCimaPolOuNaoInd} implies that the number of extensions of an embedding in 
$\mathcal{E}(H_{h-1},G) \setminus 
\mathcal{E}^{\ind}_{\clean}(H_{h-1},G)$ to embeddings from $H_h$ into $G$ is at most
$(3\varepsilon'C)n^hp^{e(H_h)}\leq 
(\varepsilon/2)n^hp^{e(H_h)}$. Therefore, using~\eqref{eq:upperCleanInd} we conclude that  
$|\mathcal{E}(H_h,G)|<(1+\varepsilon)n^hp^{e(H_h)}$.
\end{proof}

\section{Concluding remarks}\label{sec:concluding}
We say that a graph $G=(V,E)$ satisfies property $\calq(\eta,\delta,\alpha)$ if, for 
every subgraph $G[S]$ induced by 
$S\subset V$ with $|S|\geq\eta |V|$, we have
$(\alpha - \delta){|S|\choose 2} < |E(G[S])| < (\alpha + \delta){|S|\choose 2}$.
In~\cites{KoRo03,Ro86}, answering affirmatively a question posed by Erd\H{o}s (see, 
e.g.,\cite{Er79}~and~\cite{Bo04}*{p.~363}; see also~\cite{Ni01}), R\"odl 
proved that for every positive integer $m$ and for every positive $\alpha,\eta<1$ there exist $\delta>0$ and an 
integer $n_0$ 
such that, if $n\geq n_0$, then every $n$-vertex graph $G$ satisfying $\calq(\eta,\delta,\alpha)$ contains all graphs 
with $m$ vertices as induced subgraphs.
In~\cite{KoMoScTa14+}, we apply Theorem~\ref{thm:emb} to obtain a variant of this result, which allows one 
to count the
number of copies (not necessarily induced) of some fixed $3$-uniform hypergraph in hypergraphs satisfying a property similar to 
$Q(\eta,\delta,\alpha)$, as long as they are subhypergraphs of sufficiently ``jumbled'' $3$-uniform sparse hypergraphs.


\begin{bibdiv}
\begin{biblist}

\bib{Bo04}{book}{
      author={Bollob{\'a}s, B{\'e}la},
       title={Extremal graph theory},
      series={London Mathematical Society Monographs},
   publisher={Academic Press, Inc. [Harcourt Brace Jovanovich, Publishers],
  London-New York},
        date={1978},
      volume={11},
        ISBN={0-12-111750-2},
      review={\MR{506522}},
}

\bib{ChGrWi88}{article}{
      author={Chung, F.},
      author={Graham, R.},
      author={Wilson, R.},
       title={Quasi-random graphs},
        date={1989},
        ISSN={0209-9683},
     journal={Combinatorica},
      volume={9},
      number={4},
       pages={345\ndash 362},
         url={http://dx.doi.org/10.1007/BF02125347},
      review={\MR{1054011}},
}

\bib{ChGr02}{article}{
      author={Chung, Fan},
      author={Graham, Ronald},
       title={Sparse quasi-random graphs},
        date={2002},
        ISSN={0209-9683},
     journal={Combinatorica},
      volume={22},
      number={2},
       pages={217\ndash 244},
         url={http://dx.doi.org/10.1007/s004930200010},
        note={Special issue: Paul Erd{\H{o}}s and his mathematics},
      review={\MR{1909084 (2003d:05110)}},
}

\bib{ChGr08}{article}{
      author={Chung, Fan},
      author={Graham, Ronald},
       title={Quasi-random graphs with given degree sequences},
        date={2008},
        ISSN={1042-9832},
     journal={Random Structures \& Algorithms},
      volume={32},
      number={1},
       pages={1\ndash 19},
         url={http://dx.doi.org/10.1002/rsa.20188},
      review={\MR{MR2371048 (2009a:05189)}},
}

\bib{CoFoZh13+}{article}{
      author={Conlon, David},
      author={Fox, Jacob},
      author={Zhao, Yufei},
       title={Extremal results in sparse pseudorandom graphs},
        date={2014},
        ISSN={0001-8708},
     journal={Adv. Math.},
      volume={256},
       pages={206\ndash 290},
         url={http://dx.doi.org/10.1016/j.aim.2013.12.004},
      review={\MR{3177293}},
}

\bib{CoFoZh13+b}{article}{
      author={Conlon, David},
      author={Fox, Jacob},
      author={Zhao, Yufei},
       title={A relative {S}zemer\'edi theorem},
        date={2015},
        ISSN={1016-443X},
     journal={Geom. Funct. Anal.},
      volume={25},
      number={3},
       pages={733\ndash 762},
         url={http://dx.doi.org/10.1007/s00039-015-0324-9},
      review={\MR{3361771}},
}

\bib{Er79}{inproceedings}{
      author={Erd{\H o}s, Paul},
       title={Some old and new problems in various branches of combinatorics},
        date={1979},
   booktitle={{P}roc.\ 10th {S}outheastern {C}onference on {C}ombinatorics,
  {G}raph {T}heory and {C}omputing},
   publisher={Utilitas Math.},
     address={Winnipeg, Man.},
       pages={19\ndash 37},
}

\bib{FrKr12}{article}{
      author={Frieze, Alan},
      author={Krivelevich, Michael},
       title={Packing {H}amilton cycles in random and pseudo-random
  hypergraphs},
        date={2012},
        ISSN={1042-9832},
     journal={Random Structures \& Algorithms},
      volume={41},
      number={1},
       pages={1\ndash 22},
         url={http://dx.doi.org/10.1002/rsa.20396},
      review={\MR{2943424}},
}

\bib{FrKrLo12}{article}{
      author={Frieze, Alan},
      author={Krivelevich, Michael},
      author={Loh, Po-Shen},
       title={Packing tight {H}amilton cycles in 3-uniform hypergraphs},
        date={2012},
        ISSN={1042-9832},
     journal={Random Structures \& Algorithms},
      volume={40},
      number={3},
       pages={269\ndash 300},
         url={http://dx.doi.org/10.1002/rsa.20374},
      review={\MR{2900140}},
}

\bib{KoMoScTa14+}{unpublished}{
      author={Kohayakawa, Y.},
      author={Mota, G.~O.},
      author={Schacht, M.},
      author={Taraz, A.},
       title={Counting results for sparse pseudorandom hypergraphs~{II}},
        note={submitted},
}

\bib{KoRo03}{incollection}{
      author={Kohayakawa, Y.},
      author={R{\"o}dl, V.},
       title={Szemer\'edi's regularity lemma and quasi-randomness},
        date={2003},
   booktitle={Recent advances in algorithms and combinatorics},
      series={CMS Books Math./Ouvrages Math. SMC},
      volume={11},
   publisher={Springer},
     address={New York},
       pages={289\ndash 351},
         url={http://dx.doi.org/10.1007/0-387-22444-0_9},
      review={\MR{1952989 (2003j:05065)}},
}

\bib{KoRoSi04}{article}{
      author={Kohayakawa, Y.},
      author={R{\"o}dl, V.},
      author={Sissokho, P.},
       title={Embedding graphs with bounded degree in sparse pseudorandom
  graphs},
        date={2004},
        ISSN={0021-2172},
     journal={Israel J. Math.},
      volume={139},
       pages={93\ndash 137},
      review={\MR{2041225 (2004m:05243)}},
}

\bib{Ni01}{article}{
      author={Nikiforov, V.},
       title={On the edge distribution of a graph},
        date={2001},
        ISSN={0963-5483},
     journal={Combin. Probab. Comput.},
      volume={10},
      number={6},
       pages={543\ndash 555},
         url={http://dx.doi.org/10.1017/S0963548301004837},
      review={\MR{1869845}},
}

\bib{Ro86}{article}{
      author={R{\"o}dl, Vojt{\v{e}}ch},
       title={On universality of graphs with uniformly distributed edges},
        date={1986},
        ISSN={0012-365X},
     journal={Discrete Math.},
      volume={59},
      number={1-2},
       pages={125\ndash 134},
      review={\MR{837962 (88b:05098)}},
}

\bib{Th87a}{incollection}{
      author={Thomason, Andrew},
       title={Pseudorandom graphs},
        date={1987},
   booktitle={Random graphs '85 ({P}ozna\'n, 1985)},
      series={North-Holland Math. Stud.},
      volume={144},
   publisher={North-Holland},
     address={Amsterdam},
       pages={307\ndash 331},
      review={\MR{89d:05158}},
}

\bib{Th87b}{incollection}{
      author={Thomason, Andrew},
       title={Random graphs, strongly regular graphs and pseudorandom graphs},
        date={1987},
   booktitle={Surveys in combinatorics 1987 ({N}ew {C}ross, 1987)},
      series={London Math. Soc. Lecture Note Ser.},
      volume={123},
   publisher={Cambridge Univ. Press},
     address={Cambridge},
       pages={173\ndash 195},
      review={\MR{88m:05072}},
}

\end{biblist}
\end{bibdiv}
\end{document}